\newif\ifjoc\jocfalse

\ifjoc
\documentclass[joc]{ipart}
\else
\documentclass{amsart}
\fi

\usepackage{enumerate}
\usepackage{tikz}
\usetikzlibrary{graphs,positioning}
\usepackage{comment}
\usepackage{tabularx}

\ifjoc\startlocaldefs\fi
\newtheorem{theorem}{Theorem}[section]
\newtheorem{lemma}[theorem]{Lemma}
\newtheorem{corollary}[theorem]{Corollary}
\newtheorem{question}[theorem]{Question}
\newtheorem{observation}[theorem]{Observation}
\ifjoc\endlocaldefs

\pubyear{2022}
\volume{0}
\issue{0}
\firstpage{1}
\lastpage{1}
\fi

\ifjoc\else
\author{Csaba Bir\'o}
\address{Department of Mathematics, University of Louisville, Louisville KY 40292, USA}
\email{csaba.biro@louisville.edu}

\author{Beth Novick}
\address{School of Mathematical and Statistical Sciences, Clemson University, Clemson SC 29634, USA}
\email{nbeth@clemson.edu}

\author{Daniela Olejnikova}
\address{School of Mathematical and Statistical Sciences, Clemson University, Clemson SC 29634, USA}
\curraddr{Public Health Scotland, Gyle Square, 1 South Gyle Crescent, Edinburgh, EH12 9EB, United Kingdom}
\email{daniela.olejnikova@phs.scot}
\fi

\begin{document}

\ifjoc\begin{frontmatter}\fi

\title{Metric dimension of growing infinite graphs}


\begin{abstract}
We investigate how the metric dimension of infinite graphs change when we add edges to the graph. Our two main results: (1) there exists a growing sequence of graphs (under the subgraph relation, but without adding vertices) for which the metric dimension changes between finite and infinite infinitely many times; (2) finite changes in the edge set can not change the metric dimension from finite to infinite or vice versa.
\end{abstract}

\ifjoc



\end{frontmatter}
\else
\maketitle
\fi


\section{Introduction}

In this paper we consider connected graphs,  both finite and  infinite.   For a graph $G$ and $u,v\in V(G)$,  we define the distance between $u$ and $v$, denoted $d(u,v)$,  to be the number of edges in a shortest $u$--$v$ path.   A vertex $w\in V(G)$ \emph{resolves} $u$ and $v$ if $d(u,w) \not=d(v,w)$.    If $W=(w_1, \ldots,w_k)$ is an ordered set of vertices in $V(G)$,  we define the \emph{metric representation} or \emph{metric code} of $v$ with respect to $W$ to be 
$$r(v\vert W) = (d(v,w_1), \ldots,  d(v,w_k)).$$
A set $W$ is a \emph{resolving set},  if all vertices in $V(G)$ have distinct metric representations with respect to $W$.  The minimum cardinality of a resolving set is the \emph{metric dimension} of $G$,  denoted $\beta(G)$. Note that the metric dimension is a nonnegative integer for finite graphs,  but may be infinite for infinite graphs. 

The concept of metric representation was introduced by Slater \cite{Slater75} in 1975 in the context of a location problem:  The location of an intruder  is modeled by its metric representation, which Slater called its \emph{locating set}.  The minimum number of sensors to uniquely determine the location of the intruder in the network is then the metric dimension,  which Slater termed the \emph{location number}.  Harary and Melter  \cite{HM} independently introduced  \emph{metric dimension} in 1976.  Many applications of this concept,  including robot navigation \cite{KRR},  sonar \cite{Slater75},  chemistry \cite{CEJO},  network discovery \cite{Z},  and finding person zero at the beginning of an epidemic \cite{MOT}, 
have helped motivate extensive research.  See \cite{CZ} for an early survey.  

The problem of determining $\beta(G)$  is, in general, NP-hard \cite{KRR}.  Much is known for specific graphs.  For example, the only finite  graphs with metric dimension equal to $1$ are paths \cite{CEJO},  and this generalizes to infinite rays \cite{CHMPP}.  Efficient methods for finding the metric dimension of a tree were described  independently by various authors \cite{CEJO, HM, KRR, Slater75}.  C\'{a}ceres et al.\ \cite{CHMPP} studied the metric dimension of the Cartesian product of graphs.

Several variations of metric dimension have been studied.   Motivated by the observation that non-isomorphic graphs on a fixed vertex set can have identical metric codes for each vertex, Seb\H{o} and Tannier \cite{ST} introduced \emph{strong metric dimension}:  a vertex $w$ \emph{strongly} resolves $u$ and $v$ if some shortest $u$-$w$ path contains $v$, or if some shortest $v$-$w$ path contains $u$.  Sensors measuring  \emph{truncated metric dimension}   \cite{BR-V,TFL} detect only relatively `close'  objects.   
Mol,  Murphy and Oellermann \cite{MMO} considered the question of how much the metric dimension of a graph can be reduced by adding edges,  defining the \emph{threshold dimension} of a graph $G$ to be the minimum of $\beta(H)$ over all graphs $H$ on $V(G)$ for which $G$ is a subgraph.  The analogous concept for strong metric dimension is introduced by Benakli et al.\ \cite{BBDNO}.   Metric dimension of random graphs was studied by Bollob\'as, Mische, and Pra\l at \cite{BMP}.

The following theorem was proved by Chartrand, Poisson, and Zhang \cite{CPZ} for finite graphs, and by C\'aceres et al.\ \cite{CHMPP} for infinite graphs.

\begin{theorem}\label{thm:3k-1}
If $\beta(G)=k$, then $G$ has maximum degree at most $3^k-1$.
\end{theorem}

Recall that a graph $G$ is \emph{locally finite}, if every vertex in $G$ has finite degree.
Later in the paper we will use the following important corollary of Theorem~\ref{thm:3k-1}.

\begin{corollary}\label{cor:lf}
If $\beta(G)<\infty$, then $G$ is locally finite.
\end{corollary}

The issue of how  metric dimension changes with the addition of edges is an intriguing one.   Some graph parameters,  such as chromatic number and clique number are non-decreasing.  Others,  such as independence number,  diameter, and matching number are non-increasing.   The invariant $\beta(G)$,  on the other hand can increase,  decrease,  or remain the same as edges are added.  In 2015,  Eroh et al.\ \cite{EFKY} assert that the addition of one edge to a finite graph can raise the metric dimension by an arbitrary amount,  specifically it can be raised by approximately order $n$,  where $n = |V(G)|$.  In the same paper,  these authors show that the removal of a single edge can increase the metric dimension by at most $2$.    Later, in 2021, Mashkaria et al.\ \cite{MOT} give a class of examples showing that in fact the metric dimension of a finite graph can increase by an exponential amount upon the addition of an edge.  

We say that a graph parameter is \emph{stable} if there is a universal positive integer $k$ such that the addition (or
removal) of a single edge in an arbitrary graph changes the parameter by at most $k$.   Otherwise that parameter is \emph{unstable}.   In this setting we see that the results mentioned above \cite{EFKY, MOT} establish that metric dimension for finite graphs is unstable.

Our focus in this paper is to study the behavior of infinite graphs as edges are added.   In Section 2 we show that with the right choice of sequences of edge sets added,  the metric dimension will change from finite,  to infinite,  and back,  an infinite number of times.   In Section 3 we show that this behavior is only possible if each set in the sequence is infinite -- said differently,  adding (or deleting) a finite number of edges cannot change the metric dimension from finite to infinite or vice versa.   Next we note that our result in Section 3 is non-trivial because for infinite graphs, as well,  metric dimension is unstable; furthermore, we propose a new question and prove a partial result.

\section{Growing sequence with infinitely many changes}\label{sec:growing}

Throughout this paper, we use $\mathbb{N}$ to denote the set of nonnegetive integers.

\subsection{Definition of the graph sequence}

Let $V=\{v_{ij}:i\in\mathbb{N},j\in\{0,1\}\}$. We will define the edge sets $E_0, E_0', E_1,E_1'\ldots$ such that
\[
E_0\subseteq E_0'\subseteq E_1\subseteq E_1'\subseteq\cdots
\]
and the corresponding growing sequences of graphs $G_i=(V,E_i)$, and $G_i'=(V,E_i')$ for $i\in\mathbb{N}$.

The definition of the edge sets is as follows.
\begin{align*}
E_i&=\{(v_{ab},v_{cd}):|a-c|\leq i\}\\
E_i'&=E_i\cup\{(v_{ab},v_{cd}):|a-c|=i+1,\text{ and }b\neq d\}\\
\end{align*}

The graph $G_0$ is the infinite matching and it is disconnected. The graph $G_0'$ is the so-called ladder, and $\beta(G_0')=2$. See Figure~\ref{fig:threegraphs} for $G_1$, $G_1'$, and $G_2$.

\tikzset{inner sep=0mm, minimum size=2mm,every node/.style=draw,circle,on grid}

\tikzset{comb/.pic={
\node (1) {};
\node (1u) [below=of 1] {} edge (1);
\node (2) [right=of 1] {} edge (1); 
\node (2u) [below=of 2] {} edge (2);
\node (3) [right=of 2] {} edge (2);
\node (3u) [below=of 3] {} edge (3);
\node (4) [right=of 3] {} edge (3);
\node (4u) [below=of 4] {} edge (4);
\node (5) [right=of 4,draw=none] {} edge (4);
\node (5u) [below=of 5,draw=none] {};
\node [draw=none,right=of 5,yshift=-0.5cm] {$\cdots$};
}}

\tikzset{ladder/.pic={
\pic{comb};
\graph{ (1u) -- (2u) -- (3u) -- (4u) -- (5u) };
}}

\tikzset{crossed/.pic={
\pic{ladder};
\graph{ (1) -- (2u) -- (3) -- (4u) -- (5) };
\graph{ (1u) -- (2) -- (3u) -- (4) -- (5u) };
}}

\tikzset{G1p/.pic={
\pic{crossed};
\graph{ (1) -- (3u) };
\graph{ (1u) -- (3) };
\graph{ (2) -- (4u) };
\graph{ (2u) -- (4) };
\graph{ (3) -- (5u) };
\graph{ (3u) -- (5) };
}}

\tikzset{G2/.pic={
\pic{G1p};
\graph{ (1) -- [bend left] (3) };
\graph{ (1u) -- [bend right] (3u) };
\graph{ (2) -- [bend left] (4) };
\graph{ (2u) -- [bend right] (4u) };
\graph{ (3) -- [bend left] (5) };
\graph{ (3u) -- [bend right] (5u) };
}}

\begin{figure}
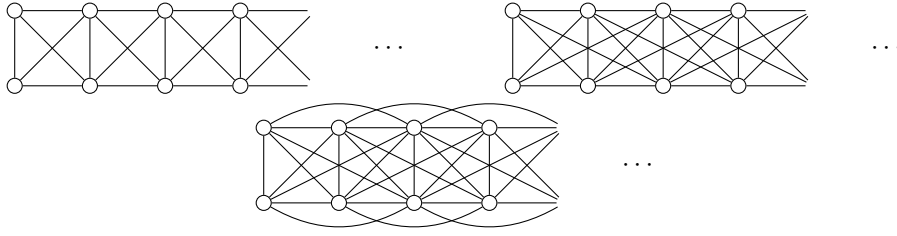

\tikz\pic{crossed};\hskip 0.5in\tikz\pic{G1p};

\tikz\pic{G2};
\caption{The graphs $G_1$, $G_1'$, and $G_2$}\label{fig:threegraphs}
\end{figure}

\begin{theorem}
For all $i\geq 1$,
\begin{enumerate}[a)]
\item $\beta(G_i)=\infty$
\item $\beta(G_i')\leq 2i+1$ 
\end{enumerate}
\end{theorem}

\begin{proof}
Let $i\geq 1$ and suppose that $\beta(G_i)<\infty$. Let $W$ be a finite resolving set. Let $k=\max\{a:v_{ab}\in W\}+1$. Then $v_{k,0}$ and $v_{k,1}$ are not distinguished by $W$. Indeed, if $v_{ab}\in W$, then $d(v_{ab},v_{k,0})=d(v_{ab},v_{k,1})=\left\lceil\frac{k-a}{i}\right\rceil$. So $W$ is not a resolving set, a contradiction.

To see the second part, let $W=\{v_{ab}:a+b\leq i\}$, and fix the ordering on $W$ as $(v_{0,0},v_{1,0},\ldots,v_{i,0},v_{0,1},v_{1,1},\ldots,v_{i-1,1})$. Then $|W|=2i+1$. We will show that $W$ is a resolving set of $G_i'$.

To see this, we introduce two functions $\alpha_i(k)$ and $\beta_i(k)$ to measure the distance of vertices in $G_i'$. Let
\begin{align*}
\alpha_i(k)&=d(v_{ab},v_{cd})\text{ if }|a-c|=k\text{ and }b=d\\
\beta_i(k)&=d(v_{ab},v_{cd})\text{ if }|a-c|=k\text{ and }b\neq d.\\
\end{align*}
For example, the sequence $\{\alpha_2(k)\}_{k=0}^\infty$ starts with 0, 1, 1, 2, 2, 2, 2, 3, 3, 4, 4, 4, 4, 5,\ldots, and $\{\beta_2(k)\}_{k=0}^\infty$ starts with 1, 1, 1, 1, 2, 2, 3, 3, 3, 3, 4, 4, 5, 5,\ldots

In general,
\begin{align*}
\alpha_i(k)&=\begin{cases}
\left\lceil\frac{k}{i+1}\right\rceil+1 & \text{\rule[-1em]{0pt}{1em}if $\frac{k}{i+1}$ is an odd integer}\\
\left\lceil\frac{k}{i+1}\right\rceil & \text{otherwise,}
\end{cases}\\
\beta_i(k)&=\begin{cases}
\left\lceil\frac{k}{i+1}\right\rceil+1 & \text{\rule[-1em]{0pt}{1em}if $\frac{k}{i+1}$ is an even integer}\\
\left\lceil\frac{k}{i+1}\right\rceil & \text{otherwise.}
\end{cases}
\end{align*}

\begin{lemma}\label{l:monotone}
For all $k\in\mathbb{N}$, $\alpha_i(k)\leq \alpha_i(k+1)$, $\beta_i(k)\leq\beta_i(k+1)$. Less obviously $\alpha_i(k)\leq\beta_i(k+1)$ and $\beta_i(k)\leq\alpha_i(k+1)$.
\end{lemma}

\begin{proof}
When $k/(i+1)$ is not an integer, the assertions are clear. If $k/(i+1)$ is an odd integer, then $\alpha_i(k)=\lceil k/(i+1)\rceil+1=\lceil (k+1)/(i+1)\rceil=\beta_i(k+1)$; the other assertions are trivial. The case when $k/(i+1)$ is an even integer is similar.
\end{proof}

\begin{lemma}\label{l:opposite}
For all $k\in\mathbb{N}$, we have $\alpha_i(k+j)\neq\beta_i(k+j)$ for some $0\leq j\leq i$.
\end{lemma}

\begin{proof}
Let $0\leq j\leq i$ be such that $(k+j)/(i+1)$ is an integer. Then it is an even integer or an odd integer. Either way
\[
\alpha_i\left(\frac{k+j}{i+1}\right)\neq\beta_i\left(\frac{k+j}{i+1}\right).\qedhere
\]
\end{proof}

\begin{lemma}\label{l:sameside}
If $\alpha_i(k)=\alpha_i(k+1)=\cdots=\alpha_i(k+i+1)$, then $\beta_i(k+i)<\beta_i(k+i+1)$. Symmetrically, if $\beta_i(k)=\beta_i(k+1)=\cdots=\beta_i(k+i+1)$, then $\alpha_i(k+i)<\alpha_i(k+i+1)$.
\end{lemma}

\begin{proof}
We will show the first assertion; the proof of the second one is similar. Suppose $\alpha_i(k)=\cdots=\alpha_i(k+i+1)$. If $(k+j)/(i+1)$ is an integer for some $1\leq j\leq i$, then either $\alpha_i(k+j-1)<\alpha_i(k+j)$, or $\alpha_i(k+j)<\alpha_i(k+j+1)$, depending on the parity of $(k+j)/(i+1)$. Either way, it would contradict the condition. Otherwise, $k/(i+1)$ and $(k+i+1)/(i+1)$ are both integers. Since $\alpha_i(k)=\alpha_i(k+i+1)$, we must have that $k/(i+1)$ is an odd integer, so $(k+i+1)/(i+1)$ is an even integer, and $(k+i)/(i+1)$ is not an integer. Hence
\[
\beta_i(k+i)=
\left\lceil\frac{k+i}{i+1}\right\rceil=
\left\lceil\frac{k+i+1}{i+1}\right\rceil<
\left\lceil\frac{k+i+1}{i+1}\right\rceil+1=
\beta_i(k+i+1).\qedhere
\]
\end{proof}

\begin{lemma}\label{l:diagonal}
For all $k\in\mathbb{N}$, we have $\alpha_i(k)<\beta_i(k+i+1)$ and $\beta_i(k)<\alpha_i(k+i+1)$.
\end{lemma}

\begin{proof}
For the first assertion, suppose $\alpha_i(k)=\beta_i(k+i+1)$. Then
\[
\left\lceil\frac{k}{i+1}\right\rceil+y=
\left\lceil\frac{k+i+1}{i+1}\right\rceil+z,
\]
where $y,z\in\{0,1\}$ with $y=1$ iff $k/(i+1)$ is an odd integer, and $z=1$ iff $(k+i+1)/(i+1)$ is an even integer. Since $\left\lceil\frac{k}{i+1}\right\rceil+1=
\left\lceil\frac{k+i+1}{i+1}\right\rceil$, we have $y=1$, and $z=0$. So $k/(i+1)$ is an odd integer. But then $(k+i+1)/(i+1)$ is an even integer, contradicting $z=0$.

The other assertion is similar.
\end{proof}

We are ready to finish the proof of $\beta(G_i')\leq 2i+1$ by showing that $W$ is a resolving set. To see this, let $v_{ab}, v_{cd}\in V$; we will show that the distance vectors of $v_{ab}$ and $v_{cd}$ are distinct. Without loss of generality, assume $c\geq a$. If $c>a$, then, by Lemma~\ref{l:monotone},
\[
r(v_{ab}|W)\leq r(v_{a+1,b}|W)\leq r(v_{cd}|W)
\]
So if $r(v_{ab}|W)=r(v_{cd}|W)$, then $r(v_{ab}|W)=r(v_{a+1,b}|W)$.

This shows that we may assume $a\leq c\leq a+1$. Up to symmetry, there are three different cases.

\emph{Case 1: $c=a$, $b=0$, $d=1$}

The distance vectors are
\begin{center}
\begin{tabular}{>{$}c<{$}>{$}c<{$}>{$}c<{$}>{$}c<{$}>{$}c<{$}>{$}c<{$}>{$}c<{$}>{$}c<{$}}
(\alpha_i(a) & \alpha_i(a-1) & \ldots & \alpha_i(a-i) & \beta_i(a) & \beta_i(a-1) & \ldots & \beta_i(a-i+1))\\
(\beta_i(a) & \beta_i(a-1) & \ldots & \beta_i(a-i) & \alpha_i(a) & \alpha_i(a-1) & \ldots & \alpha_i(a-i+1))\\
\end{tabular}
\end{center}
Due to Lemma~\ref{l:opposite}, there is a distinct coordinate somewhere among the first $i+1$.

\emph{Case 2: $c=a+1$, $b=d$}

We may assume $b=d=0$. The case $b=d=1$ is symmetric by swapping $\alpha_i$ and $\beta_i$ in the argument.

The distance vectors are
\begin{center}
\begin{tabular}{>{$}c<{$}>{$}c<{$}>{$}c<{$}>{$}c<{$}>{$}c<{$}>{$}c<{$}>{$}c<{$}>{$}c<{$}}
(\alpha_i(a) & \alpha_i(a-1) & \ldots & \alpha_i(a-i) & \beta_i(a) & \beta_i(a-1) & \ldots & \beta_i(a-i+1))\\
(\alpha_i(a+1) & \alpha_i(a) & \ldots & \alpha_i(a-i+1) & \beta_i(a+1) & \beta_i(a) & \ldots & \beta_i(a-i+2))\\
\end{tabular}
\end{center}
If the two metric codes are equal, then $\alpha_i(a-i)=\alpha_i(a-i+1)=\cdots=\alpha_i(a+1)$, so by Lemma~\ref{l:sameside}, we conclude $\beta_i(a)<\beta_i(a+1)$, a contradiction.

\emph{Case 3: $c=a+1$, $b\neq d$}

Again, by the symmetry of $\alpha_i$ and $\beta_i$, we may assume that $b=0$ and $d=1$.

The distance vectors are
\begin{center}
\begin{tabular}{>{$}c<{$}>{$}c<{$}>{$}c<{$}>{$}c<{$}>{$}c<{$}>{$}c<{$}>{$}c<{$}>{$}c<{$}}
(\alpha_i(a) & \alpha_i(a-1) & \ldots & \alpha_i(a-i) & \beta_i(a) & \beta_i(a-1) & \ldots & \beta_i(a-i+1))\\
(\beta_i(a+1) & \beta_i(a) & \ldots & \beta_i(a-i+1) & \alpha_i(a+1) & \alpha_i(a) & \ldots & \alpha_i(a-i+2))\\
\end{tabular}
\end{center}
Suppose the metric codes are equal. If $i$ is even, then $\alpha_i(a-i)=\beta_i(a-i+1)=\alpha_i(a-i+2)=\beta_i(a-i+3)=\cdots=\alpha_i(a)=\beta_i(a+1)$, which contradicts Lemma~\ref{l:diagonal}. If $i$ is odd, then $\alpha_i(a-i)=\beta_i(a-i+1)=\alpha_i(a-i+2)=\beta_i(a-i+3)=\cdots=\alpha_i(a-1)=\beta_i(a)=\alpha_i(a+1)$. So $\alpha_i(a-i)=\alpha_i(a)=\alpha_i(a+1)$. But also, $\alpha_i(a)=\beta_i(a+1)$, so $\alpha_i(a-i)=\beta_i(a+1)$, again, contradicting Lemma~\ref{l:diagonal}.
\end{proof}

\section{Finite changes in the edge set} \label{sec:stable}

In Section~\ref{sec:growing} we have shown that it is possible to change between finite and infinite metric dimension infinitely many times with the addition of edges. Notice, that in the example we provided, an infinite number of edges were added in every step. It is natural to ask if the same can be achieved with the addition of a finite number of edges.

The following two theorems and their corollary answer this question in the negative.

\begin{theorem}\label{thm:addition}
Let $G$ be a graph, and let $G'$ be a graph constructed from $G$ by adding an edge between vertices $u$ and $v$. Let $W$ be a resolving set of $G$, and let
\[
W'=W\cup\bigcup_{w\in W}\{x\in V(G):d_G(w,x)\in I(d_G(w,u),d_G(w,v))\},
\]
where $I(a,b)$ represents the closed interval of integers between the integers $a$ and $b$, regardless of which one is greater.

Then $W'$ is a resolving set of $G'$.
\end{theorem}

\begin{theorem}\label{thm:removal}
Let $G$ be a graph, and let $G'$ be a graph constructed from $G$ by removing an edge between vertices $u$ and $v$. Let $W$ be a resolving set of $G$, and let
\[
W'=W\cup\{u,v\}
\]
Then $W'$ is a resolving set of $G'$.
\end{theorem}

We note that Eroh et al.\ \cite{EFKY} proved a result similar to Theorem~\ref{thm:removal} in the context of finite graphs.   Their proof method is not much different from ours.   Furthermore,  we note that their proof works in the infinite case as well.

\begin{corollary}\label{cor:stable}
Let $G$ be a graph and let $G'$ be constructed from $G$ by changing (adding or removing) finitely many edges. Then the metric dimension of $G$ is finite if and only if the metric dimension of $G'$ is finite.
\end{corollary}
\begin{proof}
Suppose $\beta(G)<\infty$, and let $W$ be a finite resolving set. We can form a finite sequence of addition and removal of edges to construct $G'$. In each step, apply Theorem~\ref{thm:addition} or Theorem~\ref{thm:removal} to construct a new resolving set. Notice that the set we add in each step is finite. In Theorem~\ref{thm:removal} this is obvious; in Theorem~\ref{thm:addition} it follows from the fact that, by Corollary~\ref{cor:lf}, $G$ is locally finite.

Now suppose that $\beta(G)=\infty$. Suppose for a contradiction that a finite sequence of changes turns $G$ into $G'$, and $\beta(G')<\infty$. Then one can reverse the changes to construct $G$ from $G'$, contradicting the first part of the theorem.
\end{proof}

\subsection*{Proof of Theorem~\ref{thm:addition}}

Let $G$ be a graph with a resolving set $W$. Let $u,v\in V(G)$, $u\not\sim v$, and let $G'=G+uv$. Let $W'$ be the set defined in the statement of the theorem. If $W$ is a resolving set of $G'$, then $W'$ is also a resolving set of $G'$. So for the balance, we assume that $W$ is not a resolving set of $G'$.

Hence there are two vertices $x,y\in V(G')$ that are not resolved by $W$ in $G'$. However, they are resolved by $W$ in $G$, so there is a $w\in W$ such that $d_G(w,x)\neq d_G(w,y)$, but $d_{G'}(w,x)=d_{G'}(w,y)$.

For brevity of notation, we will write $d(a,b)=d_G(a,b)$ for distances in $G$, and $d'(a,b)=d_{G'}(a,b)$ for distances in $G'$. Furthermore we write $d(a)=d(w,a)$, and $d'(a)=d'(w,a)$. To reiterate previous statements with the new notations, without loss of generality
\[
d(x)<d(y)\qquad\text{and}\qquad d'(x)=d'(y).
\]

An important simple observation that we will repeatedly use is that distances in $G'$ are never longer than distances in $G$, so for all $a,b$, $d'(a,b)\leq d(a,b)$, and $d'(a)\leq d(a)$.

For the argument below, the key is the following lemma. Recall that a \emph{geodesic} between two vertices is a shortest path between them.
\begin{lemma}\label{lem:key}
Let $G$, $G'$, $u$, $v$, $x$, $y$, $w$ as before. Let $P$ be a geodesic from $w$ to $x$ in $G$, and let $w'\in V(P)$. If there exists a geodesic from $w'$ to $y$ in $G'$ that does not contain the edge $uv$, then $w'$ resolves $x$ and $y$ in $G'$. (See Figure~\ref{fig:lemma}.)
\end{lemma}

\begin{figure}
\begin{tikzpicture}[scale=0.7]

\node [label=$w$] (w) at (0,0) {};
\node [label=$w'$] (wp) at (4,2) {};
\draw (w) -- node [draw=none,above left,near end] {$P$} (wp);
\node [label=$x$] (x) at (7,3) {} edge (wp);
\node [label=left:$y$] (y) at (7,-1) {};

\draw (wp) to [out=-30, in=90] node [draw=none,above right] {$\not\ni uv$} (y);

\end{tikzpicture}
\caption{}\label{fig:lemma}
\end{figure}

\begin{proof}
Suppose that $w'$ does not resolve $x$ and $y$ in $G'$, that is $d'(w',x)=d'(w',y)$. Since there is a geodesic from $w'$ to $y$ in $G'$ that does not contain the edge $uv$, we have $d'(w',y)=d(w',y)$. Then
\[
d(y)\leq
d(w')+d(w',y)=
d(w')+d'(w',y)=
d(w')+d'(w',x)\leq
d(w')+d(w',x)
\]
Recall that $w'$ is on the geodesic $P$, so $d(w')+d(w',x)=d(x)$. We conclude that $d(y)\leq d(x)$, a contradiction.
\end{proof}

Now we proceed with the proof of the theorem. We distinguish two cases.

\subsubsection*{Case 1: $d(x)=d'(x)$ (Figure~\ref{fig:addition-1})}
\begin{figure}
\begin{tikzpicture}[scale=0.7]

\node [label=$w$] (w) at (0,0) {};
\node [label=$w'$] (wp) at (4,2) {} edge node [draw=none,above left,near start] {$P$} (w);
\node [label=below:$u$] (u) at (3,-1.6) {} edge node [draw=none,above right,near start] {$R'$} (w);
\node [label=below:$v$] (v) at (4,-2) {} edge [dashed] (u);
\node [label=right:$x$] (x) at (7,3) {} edge (wp);
\node [label=right:$y$] (y) at (7,-3) {};

\draw (wp) to [out=-90, in=90] node [draw=none,above right] {$Q'$} (y);
\draw (v) -- (y);

\end{tikzpicture}
\caption{Addition of edge, Case 1}\label{fig:addition-1}
\end{figure}
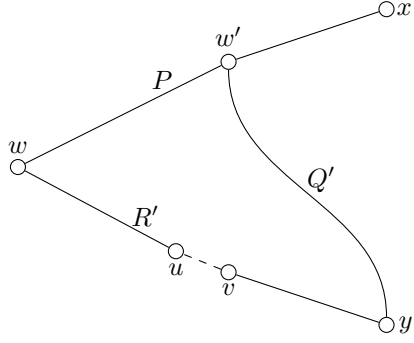

In this case we necessarily have $d'(y)<d(y)$. In other words, a geodesic $R'$ from $w$ to $y$ in $G'$ contains the edge $uv$; without loss of generality $u$ is the vertex closer to $w$ on $R'$. On the other hand, there is a geodesic $P$ from $w$ to $x$ that does not contain $uv$.

If $d(v)>d(x)$, then note that $d(x)$ must be in the interval of integers $[d(u),d(v)]$. Indeed, $d(u)<d(x)$, because $d'(x)=d'(y)$. Hence $x\in W'$, and $x$ trivially resolves $x$ and $y$.

So we may assume $d(v)\leq d(x)$. Let $w'\in V(P)$ be the unique vertex for which $d(w')=d(v)$. If $d'(w',x)\neq d'(w',y)$, then $w'$ resolves $x$ and $y$ in $G'$, so we are done. Otherwise we will argue that a geodesic $Q'$ from $w'$ to $y$ in $G'$ can not contain $uv$, and we get a contradiction by Lemma~\ref{lem:key}.

Indeed, if $\ell=d'(w',x)=d'(w',y)$, then we will prove that $v\not\in V(Q')$. For otherwise, $d'(v,y)\leq\ell$, and since $v$ is the vertex on the geodesic $R'$ that is closer to $y$, $d(v,y)=d'(v,y)\leq\ell$. Then
\[
d(y)\leq
d(v)+d(v,y)\leq
d(w')+\ell=
d(x),
\]
a contradiction.

\subsubsection*{Case 2: $d'(x)<d(x)$ and $d'(y)<d(y)$ (Figure~\ref{fig:addition-2})}
\begin{figure}
\begin{tikzpicture}[scale=0.7]

\node [label=$w$] (w) at (0,0) {};
\node [label=$w'$] (wp) at (4,2) {} edge node [draw=none,above left,near start] {$P$} (w);
\node [label=below:$u$] (u) at (3.5,0) {} edge (w);
\node [label=below:$v$] (v) at (4.5,0) {} edge [dashed] (u);
\node [label=right:$x$] (x) at (7,3) {} edge (wp);
\node [label=right:$y$] (y) at (7,-3) {};

\draw (v) to [out=0, in=-135] node [draw=none,above left] {$P'$} (x);
\draw (v) to [out=0, in=135] node [draw=none,below left] {$R'$} (y);

\end{tikzpicture}
\caption{Addition of edge, Case 2}\label{fig:addition-2}
\end{figure}
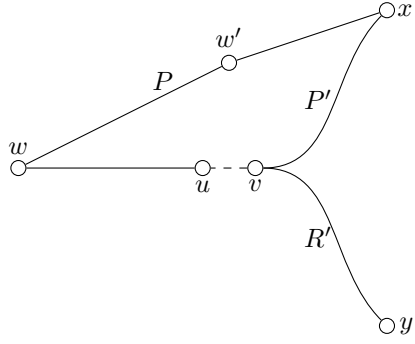

This time, there exist geodesics $P'$ from $w$ to $x$ and $R'$ from $w$ to $y$ in $G'$, both of which contain $uv$. Without loss of generality, let $u$ be the vertex on $P'$ closer to $w$. Consequently, $u$ is the vertex on $R'$ closer to $w$ as well. Note that $d'(x)=d'(y)$ implies $d(v,x)=d(v,y)$. Call this distance $\ell$. Now there is a geodesic $P$ from $w$ to $x$ in $G$ that is longer than $P'$.

If $d(v)>d(x)$, then again $d(x)\in[d(u),d(v)]$. This time $d(u)<d(x)$, because $d'(x)\leq d(x)$. Hence, again, $x\in W'$, and $x$ resolves $x$ and $y$.

So we may assume $d(v)\leq d(x)$. Let $w'$ be the unique vertex on $P$ for which $d(w')=d(v)$. Again, we will show that $w'$ resolves $x$ and $y$ in $G'$. Similarly as before, the strategy of the proof is that if we assume it does not, then we can show that $v$ is not on any geodesic from $w'$ to $y$ in $G'$, and Lemma~\ref{lem:key} finishes the proof.

As we assume that $w'$ does not resolve $x$ and $y$ in $G'$, we have $d'(w',x)=d'(w',y)$. As $w'$ is on the geodesic $P$,

\[
d(w',x)=
d(x)-d(w')<
d(y)-d(v)\leq
d(v,y)=
\ell.
\]

A quick observation is that $d'(w',x)=d(w',x)$ for otherwise there is a path from $w'$ to $x$ in $G'$ through $uv$ that is shorter than $d(w',x)$. However, $d(w',x)<\ell$, and $d(v,x)=\ell$.

So now we have seen that $d'(w',y)=d'(w',x)<\ell$, but $d'(v,y)=d(v,y)=\ell$. This shows that a geodesic from $w'$ to $y$ in $G'$ can not contain $v$, as we intended.

\subsection*{Proof of Theorem~\ref{thm:removal}}

Let $G$ be a graph with a resolving set $W$. Let $u,v\in V(G)$, $u\sim v$, and let $G'=G-uv$. Let $W'$ be the set defined in the statement of the theorem. As before, we assume that $W$ is not a resolving set of $G'$, for otherwise we are done.

Again, there are two vertices $x,y\in V(G')$ that are not resolved by $W$ in $G'$, but resolved in $G$, so there is a $w\in W$ such that (without loss of generality)
\[
d(x)<d(y)\qquad\text{and}\qquad d'(x)=d'(y).
\]

This time, distances in $G'$ are never shorter than distances in $G$, so for all $a,b$, $d(a,b)\leq d'(a,b)$, and $d(a)\leq d'(a)$.

Consider geodesics from $w$ to $x$ and from $w$ to $y$ in $G$. Let the last common vertex on them be $a$, and let the portion of these geodesics from $a$ to $x$ be called $P$, the portion from $a$ to $y$ be called $R$, and let the portion from $w$ to $a$ be called $Q$. The edge $uv$ is on one of these paths and it can not be on $R$. So again, we distinguish two cases.

\subsubsection*{Case 1: $uv$ is on $P$ (Figure~\ref{fig:removal-1})}
\begin{figure}
\begin{tikzpicture}[scale=0.7]

\node [label=below:$w$] (w) at (0,0) {};
\node [label=below:$a$] (a) at (5,0) {} edge node [draw=none,above] {$Q$} (w);
\node [label=below:$u$] (u) at (6.5,1) {} edge (a);
\node [label=below:$v$] (v) at (7.5,2) {} edge [dashed] node [draw=none,above left] {$P$} (u);
\node [label=right:$x$] (x) at (8,3) {} edge (v);
\node [label=right:$y$] (y) at (8,-3) {} edge [out=120,in=-30] node [draw=none,below left] {$R$} (a);

\end{tikzpicture}
\caption{Removal of edge, Case 1}\label{fig:removal-1}
\end{figure}
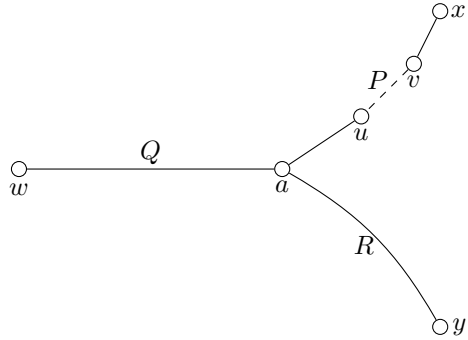

Without loss of generality, $v$ is the vertex farther from $w$. We claim $v$ resolves $x$ and $y$ in $G'$. Suppose not. Then
\[
d(v,x)=d'(v,x)=d'(v,y)\geq d(v,y),
\]
so
\[
d(x)=d(v)+d(v,x)\geq d(v)+d(v,y)\geq d(y),
\]
a contradiction.

\subsubsection*{Case 2: $uv$ is on $Q$ (Figure~\ref{fig:removal-2})}
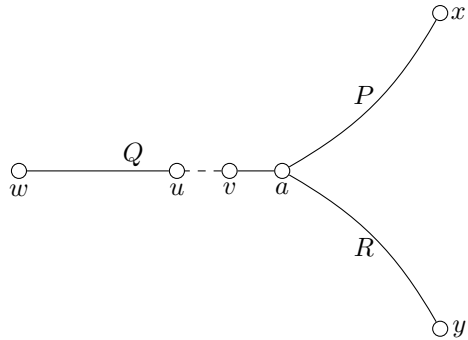
\begin{figure}
\begin{tikzpicture}[scale=0.7]

\node [label=below:$w$] (w) at (0,0) {};
\node [label=below:$u$] (u) at (3,0) {} edge node [draw=none,above,near start] {$Q$} (w);
\node [label=below:$v$] (v) at (4,0) {} edge [dashed] (u);
\node [label=below:$a$] (a) at (5,0) {} edge (v);
\node [label=right:$x$] (x) at (8,3) {} edge [out=-120,in=30] node [draw=none,above left] {$P$} (a);
\node [label=right:$y$] (y) at (8,-3) {} edge [out=120,in=-30] node [draw=none,below left] {$R$} (a);

\end{tikzpicture}
\caption{Removal of edge, Case 2}\label{fig:removal-2}
\end{figure}

Without loss of generality, $v$ is the vertex farther from $w$. We claim $v$ resolves $x$ and $y$ in $G'$. Suppose not. Then
\[
d(x)=d(v)+d(v,x)=d(v)+d'(v,x)=d(v)+d'(v,y)=d(v)+d(v,y)=d(y),
\]
a contradiction.

\section{Single edge change}

In Section~\ref{sec:stable} we have shown that changing finitely many edges can not turn the metric dimension from finite to infinite or vice versa. We have shown this by proving that a single edge change can not turn the metric dimension from finite to infinite or vice versa.

It is natural to ask, what a change of a single edge can do in terms of metric dimension. In particular, is there a bound on the change in the metric dimension after we change a single edge?

It is clear from Theorem~\ref{thm:removal} that removal of an edge can not increase the metric dimension by more than $2$. So in this section we focus on addition of an edge.

This question is particularly important, because if there is a bound (say, addition of a single edge can not change the metric dimension by more than $k$), then it would render our result from Section~\ref{sec:stable} trivial.

Of course this question is meaningful for both finite and infinite graphs. However, as we will see later, as long as there is no bound for finite graphs, we can show that there is no bound for infinite graphs. So in the next part of the section, we deal with finite graphs, and we return to the infinite case at the end of the section.

In \cite{EFKY}, the authors state the following theorem.

\begin{theorem}
There exists a graph $G$ and a non-edge $e$ in $G$ such that $\beta(G+e)-\beta(G)$
can be arbitrarily large.
\end{theorem}

While the idea in their paper seems correct, the short proof (which is an example graph $G$) is not correct.

This issue seems to have been unnoticed. Mashkaria, \'Odor, and Thiran~\cite{MOT} cite the paper \cite{EFKY}, and claim an improved example, which increases the metric dimension exponentially. More precisely, they prove the following theorem.

\begin{theorem}\label{thm:MOT}
For all $k\geq 3$ there exists a graph $G$ and a non-edge $e$ in $G$ such that $\beta(G)\leq k$ and $\beta(G+e)\geq 2^{k-1}-2$.
\end{theorem}

Since they provide a correct proof, this already ensures that our Theorems~\ref{thm:addition} and \ref{thm:removal} and Corollary~\ref{cor:stable} are meaningful. However, this result raises another interesting question.

\begin{question}\label{q:bound}
Is there a function $f(d)$ such that $\beta(G+e)\leq f(\beta(G))$ for all $G$ and $e\not\in E(G)$? If there is one, what is the smallest such function? More precisely, determine
\[
\iota(d)=\max\{\beta(G+e): \beta(G)=d, e\not\in E(G)\}.
\]
\end{question}

We know $\iota(1)=2$, but all the other values (or indeed their existence) are unknown. Theorem~\ref{thm:MOT} shows that $\iota(d)=\Omega(2^d)$.

We first present an example (without proof) of a relatively simple graph for which adding an edge roughly doubles the metric dimension. This example simplifies and corrects ideas from \cite{EFKY}. See Figure~\ref{fig:kites}.

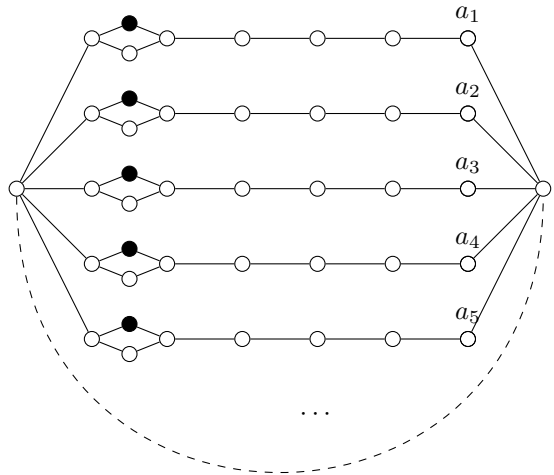
\begin{figure}
\begin{tikzpicture}
\node (u) at (0,0) {};

\node (a1) at (1,2) {} edge (u);
\node (b1) at (1,1) {} edge (u);
\node (c1) at (1,0) {} edge (u);
\node (d1) at (1,-1) {} edge (u);
\node (e1) at (1,-2) {} edge (u);

\node [fill] (a2) at (1.5,2.2) {} edge (a1);
\node (a2p) at (1.5,1.8) {} edge (a1);
\node [fill] (b2) at (1.5,1.2) {} edge (b1);
\node (b2p) at (1.5,0.8) {} edge (b1);
\node [fill] (c2) at (1.5,0.2) {} edge (c1);
\node (c2p) at (1.5,-0.2) {} edge (c1);
\node [fill] (d2) at (1.5,-0.8) {} edge (d1);
\node (d2p) at (1.5,-1.2) {} edge (d1);
\node [fill] (e2) at (1.5,-1.8) {} edge (e1);
\node (e2p) at (1.5,-2.2) {} edge (e1);

\node (a3) at (2,2) {} edge (a2) edge (a2p);
\node (b3) at (2,1) {} edge (b2) edge (b2p);
\node (c3) at (2,0) {} edge (c2) edge (c2p);
\node (d3) at (2,-1) {} edge (d2) edge (d2p);
\node (e3) at (2,-2) {} edge (e2) edge (e2p);

\foreach \x in {4,...,7} {
  \pgfmathparse{\x-1}
  \node (a\x) at (\pgfmathresult,2) {} edge (a\pgfmathresult);
  \node (b\x) at (\pgfmathresult,1) {} edge (b\pgfmathresult);
  \node (c\x) at (\pgfmathresult,0) {} edge (c\pgfmathresult);
  \node (d\x) at (\pgfmathresult,-1) {} edge (d\pgfmathresult);
  \node (e\x) at (\pgfmathresult,-2) {} edge (e\pgfmathresult);}

\node[label=$a_1$] (a7) at (6,2) {} ;
\node[label=$a_2$] (b7) at (6,1) {} ;
\node[label=$a_3$] (c7) at (6,0) {} ;
\node[label=$a_4$] (d7) at (6,-1) {} ;
\node[label=$a_5$] (e7) at (6,-2) {} ;

\node (v) at (7,0) {} edge (a7) edge (b7) edge (c7) edge (d7) edge (e7); 

\node [draw=none,below=of e5] {$\cdots$};

\draw [dashed] (u) .. controls (0,-5) and (7,-5) .. (v);

\end{tikzpicture}
\caption{The filled vertices form a resolving set of the graph without the dashed edge. Once the dashed edge is added, the vertices on the right ($a_1$ to $a_5$) are not resolved. The metric dimension roughly doubles with the addition of the edge.}\label{fig:kites}
\end{figure}

Next, we add our contribution to Question~\ref{q:bound}. The following theorem is a strengthening of Theorem~\ref{thm:MOT}.

\begin{theorem}\label{thm:nonbinary}
For all $k\geq 4$ there exist a graph $G$ and $e\not\in E(G)$ such that $\beta(G)\leq k$, but $\beta(G+e)\geq (k+1)2^{k-2}-1$. Hence $\iota(d)=\Omega(d2^d)$.
\end{theorem}

Before we can get to the proof we propose a related problem in enumerative combinatorics.

Let $T_n$ be the set of ternary strings of length $n$, consisting of digits $0$, $1$, and $2$. For $x\in T_n$, let $x(i)$ denote the $i$th digit of $x$. We say two distinct elements $x,y\in T_n$ \emph{conflict}, if there exists an $i$ with $x(i)=y(i)=2$, and for all $j$ for which $x(j)\neq y(j)$, we have $\{x(j),y(j)\}=\{0,2\}$.

We call a subset $S\subseteq T_n$ \emph{conflict-free}, if no two elements in $S$ conflict. Let $t_n$ be the size of a largest conflict-free subset of $T_n$.

\begin{observation} $t_n\geq 2^n+n2^{n-1}$
\end{observation}

\begin{proof}
The set of strings that contain at most one $2$ is clearly conflict-free.
\end{proof}

We note that one can prove that $t_n=2^n+n2^{n-1}$, but the proof is more complicated, and we only need the easy lower bound.

\begin{proof}[Proof of Theorem~\ref{thm:nonbinary}]
Let $d=k-1$. We will define a graph $G$ and $e\not\in E(G)$ such that $\beta(G)\leq d+1$, but $\beta(G+e)\geq t_d-1\geq (d+2)2^{d-1}-1$.

Let $G$ be defined as follows. Start with $t_d$ disjoint copies of $P_4$ (path with 4 edges) with endpoints $a_i$ and $b_i$ ($i=1,\ldots,t_d$). We will call these paths ``pages''. Add a vertex $w_0$ and join it to each $a_i$. Add vertices $w_1,\ldots,w_d$ that will be connected to various $b_i$'s via paths of length $1$ or $2$, as described below. The vertices $w_1,\ldots,w_d$ will be called ``digits''.

Reindex the elements $b_1,\ldots b_{t_d}$ with a set of conflict-free elements of $T_d$. So the new index of $b_i$ is a ternary string $x$. Connect $b_i$ to $w_j$ via a path of length $x(j)$, if $x(j)>0$. (If $x(j)=0$, no connection is made.) These connecting paths will be called ``ramps''.

Finally, add a vertex $c$ and connect it to each of $w_1,\ldots,w_d$. See Figure~\ref{fig:nonbinary}.

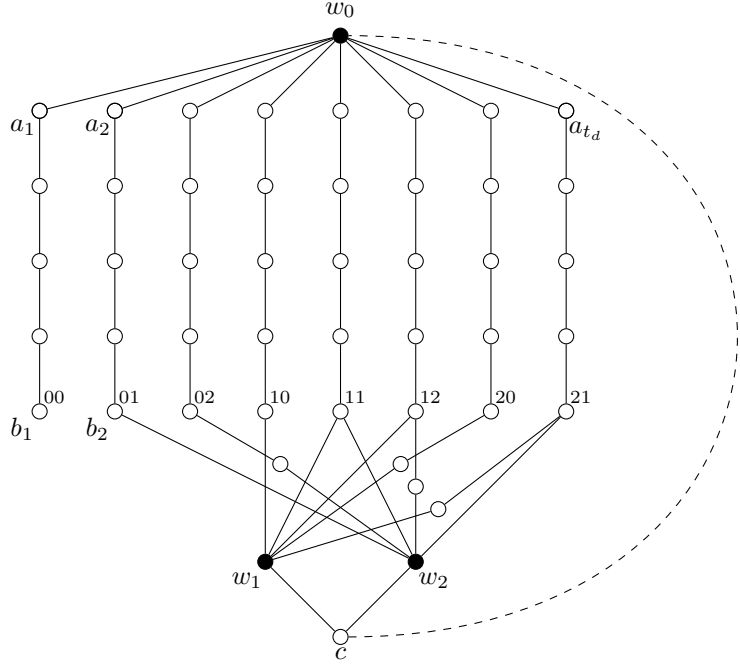
\begin{figure}
\begin{tikzpicture}
\node [fill,label=$w_0$] (w0) at (0,0) {};

\foreach \i in {1,...,8} {
  \pgfmathparse{\i-5}
  \node (p\i0) at (\pgfmathresult,-1) {} edge (w0);
  \foreach \j in {1,...,3} {
    \pgfmathparse{int(\j-1)}
    \node (p\i\j) at (\i-5,-\j-1) {} edge (p\i\pgfmathresult);
  }
}

\node [label=below left:$a_1$] at (-4,-1) {};
\node [label=below left:$a_2$] at (-3,-1) {};
\node [label=below right:$a_{t_d}$] at (3,-1) {};

\node [label=above right:{$\scriptstyle 00$},label=below left:{$b_1$}] (b00) at (-4,-5) {} edge (p13);
\node [label=above right:{$\scriptstyle 01$},label=below left:{$b_2$}] (b01) at (-3,-5) {} edge (p23);
\node [label=above right:{$\scriptstyle 02$}] (b02) at (-2,-5) {} edge (p33);
\node [label=above right:{$\scriptstyle 10$}] (b10) at (-1,-5) {} edge (p43);
\node [label=above right:{$\scriptstyle 11$}] (b11) at (0,-5) {} edge (p53);
\node [label=above right:{$\scriptstyle 12$}] (b12) at (1,-5) {} edge (p63);
\node [label=above right:{$\scriptstyle 20$}] (b20) at (2,-5) {} edge (p73);
\node [label=above right:{$\scriptstyle 21$}] (b21) at (3,-5) {} edge (p83);

\node (r02) at (-0.8,-5.7) {} edge (b02);
\node (r12) at (1,-6) {} edge (b12);
\node (r20) at (0.8,-5.7) {} edge (b20);
\node (r21) at (1.3,-6.3) {} edge (b21);

\node [fill,label=below left:$w_1$] (w1) at (-1,-7) {} edge (b10) edge (b11) edge (b12) edge (r20) edge (r21);
\node [fill,label=below right:$w_2$] (w2) at (1,-7) {} edge (b01) edge (b11) edge (b21) edge (r02) edge (r12);

\node [label=below:$c$] (c) at (0,-8) {} edge (w1) edge (w2);

\draw [dashed] (c) .. controls (7,-8) and (7,0) .. (w0);

\end{tikzpicture}
\caption{The filled vertices form a resolving set of the graph without the dashed edge. Once the dashed edge is added, every page needs a resolving vertex, except maybe one.}\label{fig:nonbinary}
\end{figure}

We claim that $\{w_0,\ldots,w_d\}$ is a resolving set of $G$.

First we show that if $v\in V(G)$ is on one of the pages, its distance vector is unique. The position of $v$ on a page (and whether it is on a page) is determined by its distance from $w_0$. The page index is then determined by the digits.

Since $c$ is of distance $1$ from each digit, we just have to show that it is distinguished from each vertex of the same property. There is (at most) one other vertex with this property: it is the $b_i$ whose new index is $11\ldots 1$. They are distinguished by their distance from $w_0$.

The only vertices remaining are the midpoints of the ramps. They are distinguished from the rest of the vertices by the fact that they are all of distance $6$ from $w_0$, and the only other vertices of distance $6$ from $w_0$ are the digits.

However, we need to be careful here. The danger is that two of these vertices have the same distance vector. To understand why this can't happen, let us study their distance vectors. For $j=0,\ldots,d$, we will refer to the distance from $w_j$ as the $j$th coordinate of the distance vector.

Let $v$ be a midpoint of a ramp that connects $w_i$ to a vertex indexed with the ternary string $x$. By definition, $x(i)=2$. For $j\geq 1$, the $j$th coordinate of the distance vector is
\[
\begin{cases}
1&\text{if $j=i$};\\
2&\text{if $j\neq i$ and $x(j)=1$};\\
3&\text{if $j\neq i$ and $x(j)\in\{0,2\}$}.\\
\end{cases}
\]

Let $u,v$ be two midpoints of ramps, and suppose they have the same distance vectors. Suppose $u$ is on a ramp of $w_i$ and $v$ is on a ramp of $w_k$. Then the $i$th coordinate of the distance vector of $u$ is $1$, while the $i$th coordinate of the distance vector of $v$ is $1$ only if $i=k$. So we have shown that $u$ and $v$ are both on a ramp of $w_i$.

Now suppose the other ends of the ramps are indexed with ternary strings $x$ and $y$. From the fact that the distance vectors of $u$ and $v$ are identical, it follows that wherever $x$ and $y$ differ, one of them has a $0$, and the other has a $2$. Together with the fact that $x(i)=y(i)=2$, this shows that $x$ and $y$ conflict, which is a contradiction.

It remains to be shown that adding an edge to $G$ raises the metric dimension to at least $t_d-1$. Indeed, if we add the edge $cw_0$, the vertices $a_i$ can not be distinguished by any of the vertices that are not on the pages. Furthermore, if there are two pages $i,j$ with no resolving vertex on either one, then $a_i$ and $a_j$ can not be distinguished. So we need a vertex in the resolving set from each page, except perhaps one. That shows $\beta(G+e)\geq t_d-1$.
\end{proof}

As mentioned above, one can prove that $t_d=2^d+d2^{d-1}$. So these techniques can not be used to further improve the theorem.

Finally, as promised, we return to the discussion of infinite graphs.

\begin{corollary}\label{cor:infiniteunstable}
For all $k\in\mathbb{N}$, there exists an infinite graph $H$, and nonadjacent vertices $u,v$  of $H$ such that $\beta(H+uv)-\beta(H)>k$.  
\end{corollary}

\begin{proof}

We begin by proving a lemma that relates the metric dimension of finite graphs with the metric dimension of certain infinite graphs constructed from them.

\begin{lemma}\label{lem:tail}
Let $G$ be a finite graph on the vertex set $V=\{v_1,\ldots,v_n\}$. Let $H$ be the countably infinite graph constructed from $G$ by adding the vertex set $U=\{u_1,u_2,\ldots\}$, and edges $v_1u_1$, $u_1u_2$, $u_2u_3$, \dots. Then
\[
\beta(G)\leq\beta(H)\leq\beta(G)+2
\]
\end{lemma}

\begin{proof}
Let $W$ be a smallest resolving set of $G$. We claim that $W\cup\{v_1,u_1\}$ is a resolving set of $H$. To see this, let $x,y\in V(H)$. If both $x,y\in V$, then already $W$ resolves them. If both are in $U$, then $v_1$ resolves them. If $x\in V$, and $y\in U$, then $d(x,v_1)<d(x,u_1)$, and $d(y,v_1)>d(y,u_1)$ so the set $\{v_1,u_1\}$ resolves them. This shows $\beta(H)\leq \beta(G)+2$.

Now let $W$ be a smallest resolving set of $H$. If $W\cap V$ is a resolving set of $G$, then $\beta(G)\leq\beta(H)$, and we are done.

Now suppose $W\cap V$ is not a resolving set of $G$. We claim $(W\cap V)\cup\{v_1\}$ is a resolving set of $G$. To see this, let $x,y\in V=V(G)$. Since $W\cap V$ does not resolve them, it means $W$ includes some vertices of $U$. However, $d(x,u_i)=d(x,v_1)+i$, and the same holds for $y$. So if $d(x,u_i)\neq d(y,u_i)$, then $d(x,v_1)\neq d(y,v_1)$, and therefore $x$ and $y$ are resolved by $v_1$. Since $|(W\cap V)\cup\{v_1\}|\leq W$, we conclude $\beta(G)\leq\beta(H)$.
\end{proof}

Now we are ready to prove Corollary~\ref{cor:infiniteunstable}. Let $k$ be a positive integer. By Theorem~\ref{thm:nonbinary} there exists a graph $G$ and $e=uv$, a non-edge of $G$, such that $\beta(G+uv)-\beta(G)>k+2$. Construct $H$ from $G$ as in Lemma~\ref{lem:tail}. By the lemma, and the assumption on $uv$,
\[
\beta(H+uv)-\beta(H)>k.
\]
\end{proof}


\ifjoc
\else
\bibliographystyle{abbrv}
\bibliography{metric}

\begin{thebibliography}{10}

\bibitem{BR-V}
A.~F. Beardon and J.~A. Rodr\'{i}guez-Vel\'{a}zquez.
\newblock On the $k$-metric dimension of metric spaces.
\newblock {\em Ars Math. Contemp.}, 16:25 -- 38, 2019.

\bibitem{Z}
Z.~Beerliova, F.~Eberhard, T.~Erlevah, A.~Hall, M.~Hoffmann, M.~Mihal\'{a}k,
  and L.~S. Ram.
\newblock Network discovery and verification.
\newblock {\em IEEE Journal on selected areas in communications},
  24(12):2168--2181, 2006.

\bibitem{BBDNO}
N.~Benakli, N.~Bong, S.~M. Dueck, L.~Eroh, B.~Novick, and O.~R. Oellerman.
\newblock The threshold strong dimension of a graph.
\newblock {\em Discrete Math.}, 344(7):1 -- 16, 2021.

\bibitem{BMP}
B.~Bollob\'{a}s, D.~Mische, and P.~Pra{\l}at.
\newblock Metric dimension for random graphs.
\newblock {\em The Electronic Jounal of Combinatorics}, 20(4):1 -- 19, 2013.

\bibitem{CHMPP}
J.~C\'{a}ceres, C.~Hernando, M.~Mora, I.~M. Pelayo, and M.~L. Puertas.
\newblock On the metric dimension of infinite graphs.
\newblock {\em Discrete Applied Math.}, 160(18):2618--2626, 2012.

\bibitem{CEJO}
G.~Chartrand, L.~Eroh, M.~A. Johnson, and O.~R. Oellermann.
\newblock Resolvability in graphs and the metric dimension of a graph.
\newblock {\em Discrete Applied Math.}, 105:99 -- 113, 2000.

\bibitem{CPZ}
G.~Chartrand, C.~Poisson, and P.~Zhang.
\newblock Resolvability and the upper dimension of graphs.
\newblock {\em J. Comput. Math. Appl.}, 39:19--28, 2000.

\bibitem{CZ}
G.~Chratrand and P.~Zhang.
\newblock The theory and applications of resolvability in graphs, a survey.
\newblock {\em Congressus Numerantium}, 160:47 -- 68, 2003.

\bibitem{EFKY}
L.~Eroh, P.~Feit, C.~X. Kang, and E.~Yi.
\newblock The effect of vertex or edge deletion on the metric dimension of
  graphs.
\newblock {\em J. Comb.}, 6(4):433--444, 2015.

\bibitem{HM}
F.~Harary and R.~A. Melter.
\newblock On the metric dimension of a graph.
\newblock {\em Ars Combin.}, 2:191--195, 1976.

\bibitem{KRR}
S.~Khuller, B.~Raghavachari, and A.~Rosenfeld.
\newblock Landmarks in graphs.
\newblock {\em Discrete Appl. Math.}, 70:217 -- 229, 1996.

\bibitem{MOT}
S.~Mashkaria, G.~\'Odor, and P.~Thiran.
\newblock On the robustness of the metric dimension of grid graphs to adding a
  single edge.
\newblock arXiv:2010.11023v2, November 2021.

\bibitem{MMO}
L.~Mol, M.~J.~H. Murphy, and O.~R. Oellermann.
\newblock The threshold dimension of a graph.
\newblock {\em Discrete Appl. Math.}, 287:118 -- 133, 2020.

\bibitem{ST}
A.~Seb\H{o} and E.~Tannier.
\newblock On metric generators of graphs.
\newblock {\em Math. Oper. Res.}, 29(2):383 -- 393, 2004.

\bibitem{Slater75}
P.~J. Slater.
\newblock Leaves of trees.
\newblock {\em Congr. Numer.}, 14:549 -- 559, 1975.

\bibitem{TFL}
R.~M. Tillquist, R. C. ad~Frongillo and M.~E. Lladser.
\newblock Truncated metric dimension for finite graphs.
\newblock arXiv:2106.14314v1, 2021.

\end{thebibliography}
\fi

\end{document}